\documentclass[12pt]{article}
\usepackage[margin=1in]{geometry}
\usepackage[utf8]{inputenc}
\usepackage[T1]{fontenc}
\usepackage{amssymb}
\usepackage{amsmath}
\usepackage{amsthm}
\usepackage{amsfonts}
\usepackage{esint}
\usepackage{dsfont}
\usepackage{wasysym}
\usepackage{bbm}
\usepackage[all]{xy}
\usepackage{hyperref}
\usepackage{caption}
\usepackage{enumitem}
\usepackage{graphicx}
\usepackage{tikz-cd}
\usepackage{multirow}
\usepackage{array}
\usepackage{parskip}
\usepackage{lmodern}
\newcolumntype{P}[1]{>{\centering\arraybackslash}p{#1}}\newcolumntype{M}[1]{>{\centering\arraybackslash}m{#1}}

\usetikzlibrary{decorations.markings,decorations.pathreplacing,positioning,cd}

\newtheorem{theorem}{Theorem}[section]
\newtheorem{corollary}[theorem]{Corollary}
\newtheorem{lemma}[theorem]{Lemma}
\newtheorem{proposition}[theorem]{Proposition}
\newtheorem{claim}{Claim}

\theoremstyle{definition}
\newtheorem{definition}[theorem]{Definition}

\newtheorem{assumption}[theorem]{Assumption}

\theoremstyle{remark}
\newtheorem*{remark}{Remark}

\newcommand{\R}{\mathbb{R}}
\newcommand{\Z}{\mathbb{Z}}

\newcommand{\Prob}{\mathbb{P}}
\newcommand{\Proba}[1]{ \Prob\left( #1 \right)}

\newcommand{\norm}[2]{\left\Vert #1 \right\Vert_{#2}}

\newcommand{\scal}[2]{\left<#1,#2\right>}

\newcommand{\chem}{\text{chem}}

\title{Fractal behavior for nodal lines of smooth planar Gaussian fields at criticality}
\author{David Vernotte}	

\begin{document}

\maketitle

\begin{abstract}
This paper is devoted to the study of the large scale geometry of the excursion set and nodal set of a planar smooth Gaussian field at criticality $\ell=\ell_c=0$. We prove that there exists $s_1>1$ such that with high probability, macroscopic nodal lines in a box of size $\lambda$ are of length at least $\lambda^{s_1}$. As an application, on the event that a box is crossed by a nodal line, then the shortest crossing is of length at least $\lambda^{s_1}$. We also prove that there exists $s_2<2$ such that with high probability, the shortest crossing is non degenerated, that is, its length is at most $\lambda^{s_2}$. The argument for the lower bound is based on a celebrated paper of Aizenman and Burchard \cite{AB98} that provides a general argument to show that random curves present a fractal behavior. For the upper bound, our proof relies on the polynomial decay of the probability of one-arm events which was proven in \cite{BG16}.
\end{abstract}

\tableofcontents

\section{General introduction}

In this paper, we study some geometric properties of the excursion and nodal sets of some smooth random Gaussian fields in the Euclidean space of dimension $2$. Gaussian percolation can be seen as the continuous analogue of classical discrete Bernoulli percolation. Instead of a random configuration $\omega : \Z^2 \to \{0,1\}$, one consider a random continuous (in fact $\mathcal{C}^1$) function called the \textit{random field}
\begin{equation}
    f : \R^2 \to \R.
\end{equation}
In the following the random field $f$ is a continuous, stationary Gaussian field, that is centered. That is we assume that:
\begin{itemize}
    \item for all $n\geq 1$, $x_1,\dots,x_n \in \R^2$ distinct points, the random vector $(f(x_1),\dots,f(x_n))$ is a non-degenerated Gaussian vector that is centered,
    \item for all $x\in \R^2$, $f(\cdot)$ and $f(\cdot+x)$ have the same law (stationarity),
    \item almost surely, the function $f : \R^2 \to \R$ is continuous.
\end{itemize}
\begin{remark}
 The third condition simply means that there exists an event of probability $1$ under which the function $f$ is continuous (we may then consider a modification of the underlying probability space to assume that each realisation of $f$ is continuous).
\end{remark}
When considering Gaussian fields, one object of interest is the \textit{covariance kernel} since it completely characterizes the law of the field.
\begin{definition}
Let $f$ be a centered Gaussian field on $\R^2$. The \textit{covariance kernel} associated to $f$ is the function
$$\begin{array}{ccccc}
    K & : & \R^2\times \R^2 & \to & \R  \\
     & & (x,y) & \mapsto & \mathbb{E}[f(x)f(y)].
\end{array}$$
\end{definition}
Since the fields we consider are stationary, their covariance kernels only depend on $x-y$, as such we can define the \textit{covariance function}.
\begin{definition}
Let $f$ be a centered, stationary Gaussian field on $\R^2$. The \textit{covarariance function} associated to $f$ is the function
$$\begin{array}{ccccc}
    \kappa & : & \R^2 & \to &\R \\
    & & x&\mapsto & K(x,0) = \mathbb{E}[f(x)f(0)].
\end{array}$$
\end{definition}
This covariance function is enough to characterize the law of $f$ since we have $K(x,y)=\kappa(x-y).$
An important example of a continuous stationary centered Gaussian field is the Bargmann-Fock field. Given a collection $(a_{i,j})_{i,j\geq 0}$ of independent identically distributed standard Gaussian random variables, we define
$$\forall x \in \R^2,\ f(x):= e^{-\frac{1}{2}\norm{x}{}^2}\sum_{i,j\geq 0}a_{i,j}\frac{x_1^ix_2^j}{\sqrt{i!j!}},$$
where $\norm{\cdot}{}$ denotes the Euclidean norm of $\R^2$.
Note that almost surely the above function is well defined for all $x\in \R^2$. Furthermore, the function $f$ obtained is almost surely analytic (in particular continuous). The field $f$ is clearly a Gaussian field and is stationary since its covariance kernel is given by $K(x,y) =\mathbb{E}[f(x)f(y)]=e^{-\frac{1}{2}\norm{x-y}{}^2}.$

We present a general way to define continuous stationary Gaussian fields. This is done via the white noise representation of such field.
\begin{definition}
 A white noise on $\R^2$ is a centered Gaussian field $W$ indexed by the functions of $L^2(\R^2)$ such that for any $\varphi_1,\varphi_2 \in L^2(\R^2)$ we have
\begin{equation}
    \mathbb{E}[W(\varphi_1)W(\varphi_2)]=\int_{\R^2}\varphi_1(x)\varphi_2(x)dx.
\end{equation}
 \end{definition}
The construction of the white noise can be made using an Hilbert basis of $L^2(\R^2)$ we refer the reader to \cite{Jan97} for more details.
Let $q\in L^2(\R^2)$ a function such that $\forall x\in \R^2, q(x)=q(-x).$ We define the stationary Gaussian field
\begin{equation}
    \label{eq:defqstarW}
    f:= q\ast W,
\end{equation}
where $\ast$ denotes convolution. That is for $x\in \R^2$ we set
$$f(x) := W(q(x-\cdot)).$$
It can be checked by a simple computation that the field $f$ obtained is a stationary and centered Gaussian field whose covariance function is given by $\kappa(x)=(q\ast q)(x).$ Moreover, under some regularity assumptions on $q$ (see Assumption \ref{a:a2}), the field $f$ obtained will be continuous.
\begin{remark}
 By setting $q(x) = \sqrt{\frac{2}{\pi}}e^{-\norm{x}{}^2}$, then the field $f=q\ast W$ obtained is the Bargmann-Fock field previously introduced.
\end{remark}
We have already mentioned that the function $q$ relates to the covariance function $\kappa$ via the equality $\kappa=q\ast q$. Another relation between $\kappa$ and $q$ can be seen via the spectral measure of the field. Assume that $\kappa$ is continuous, then by Bochner's theorem there exists a measure $\mu$ on $\R^2$ called the \textit{spectral measure} which is the Fourier transform of $\kappa$, that is, for all $x\in \R^2$,
\begin{equation}
    \kappa(x) = \int_{\R^2}e^{2i\pi\scal{x}{y}}\mu(dy),
\end{equation}
where $\scal{x}{y}$ denotes the Euclidean scalar product between $x$ and $y$. Assume also, that $\mu$ is absolutely continuous with respect to the Lebesgue measure on $\R^2$ and we denote by $\rho^2$ its density, which is called the \textit{spectral density}. We have
 $$\kappa(0)=\int_{\R^2}\mu(dy) = \int_{\R^2}\rho(y)^2dy.$$
 This shows that $\rho$ is in $L^2(\R^2)$. The existence of the spectral density is important as it implies that the law of $f$ (which is invariant by the translations) is ergodic (see \cite[Appendix B]{NS16}). Since $\rho$ is in $L^2(\R^2)$ one can consider its Fourier transform $\mathcal{F}(\rho)\in L^2(\R^2)$. It turns out that if one defines $q:= \mathcal{F}(\rho)$ then the field $q\ast W$ obtained has covariance function $\kappa$, giving a white noise decomposition of the field $f$. This stems from the fact that $\kappa =q\ast q$ and taking Fourier transform in this equality.
 
We now state some classical assumptions that we make on the function $q$.

\begin{assumption}[Symmetry]
\label{a:a1}
The function $q$ verifies the following.
\begin{itemize}
    \item $q$ is in $L^2(\R^2)$
    \item $q$ is invariant by rotation by $\frac{\pi}{2}$ about the origin and is invariant by sign change of the coordinates.
\end{itemize}
\end{assumption}

\begin{assumption}[Regularity, depends on $\alpha \in \mathbb{N}$]
\label{a:a2}
The function $q$ is in $\mathcal{C}^\alpha(\R^2)$. Moreover, for all $(\alpha_1,\alpha_2) \in \mathbb{N}^2$ with $\alpha_1+\alpha_2\leq \alpha$, the partial derivative $\partial^{\alpha_1,\alpha_2} q$ is in $L^2(\R^2).$
\end{assumption}
\begin{assumption}[Positivity of correlations]
\label{a:a3}
One of the following (to be specified) is verified
\begin{itemize}
    \item (Weak positivity) $\kappa = q\ast q\geq 0.$
    \item (Strong positivity) $q \geq 0.$
\end{itemize}
\end{assumption}
\begin{assumption}[Decay of correlations, depends on $\beta>0$]
\label{a:a4}
As $\norm{x}{}$ goes to infinity we have $$\max\left(q(x),\norm{\nabla q(x)}{}\right)=O\left(\norm{x}{}^{-\beta}\right).$$
\end{assumption}

We briefly comment on those assumptions, the reader may also refer to \cite{NS16}, \cite{MV20}, \cite{RV19} for other comments on these hypotheses. The regularity Assumption \ref{a:a2} is a technical assumption that makes the field $f=q\star W$ a continuous field (in fact $\mathcal{C}^2$ as soon as $m\geq 3$). Assumption \ref{a:a3} implies that the field possesses the FKG property. That is two increasing events (events that are favored by an increase of the values of the field) are positively correlated. Finally, the purpose of Assumption \ref{a:a4} is to replace the hypothesis of independence in classical Bernoulli percolation.
\begin{remark}
 Note that the Bargmann-Fock field satisfies Assumption \ref{a:a1}, \ref{a:a2} for all $\alpha \geq 0$, \ref{a:a3} (strong positivity) and \ref{a:a4} for all $\beta>0$. This is due to the explicit formula $q(x)=\sqrt{\frac{2}{\pi}}e^{-\norm{x}{}^2}$ which straightforwardly verifies all of the above hypotheses.
\end{remark}

We now introduce the percolation model associated to this random field. One introduce a continuous parameter called a \textit{level} which is usually denoted by $\ell \in \mathbb{R}$. We define the following random subsets of $\R^2$. 
\begin{definition}
The \textit{excursion set at level $\ell$} is denoted by $\mathcal{E}_\ell(f)$ and is defined as
\begin{equation}
    \mathcal{E}_\ell(f) := \{x\in \R^2\ |\ f(x)\geq -\ell\}.
\end{equation}
The \textit{nodal set at level $\ell$} is denoted by $\mathcal{Z}_\ell(f)$ and is defined as
\begin{equation}
    \mathcal{Z}_\ell(f) := \{x\in \R^2\ |\ f(x)=-\ell\}.
\end{equation}
\end{definition}
A main object of interest in percolation theory is the probability of the existence of \textit{crossings} in the sets $\mathcal{E}_\ell(f)$ and $\mathcal{Z}_\ell(f)$.
\begin{definition}
\label{def:cross}
Let $\mathcal{R}\subset \R^2$ be a non degenerated rectangle of length $a$ and height $b$. Let $A\subset \R^2$ be a subset. We say that $\mathcal{R}$ is crossed by $A$ in the length direction if there exists a connected component of $A\cap \mathcal{R}$ that intersects both sides of $\mathcal{R}$ of length $b$ (the small sides) and we denote $\text{Cross}_A(\mathcal{R})$ this event. In particular, $\text{Cross}_{\mathcal{E}_\ell(f)}(\mathcal{R})$ (resp $\text{Cross}_{\mathcal{Z}_\ell(f)}(\mathcal{R})$) denotes the event that $\mathcal{R}$ is crossed in the length direction by $\mathcal{E}_\ell(f)$ (resp $\mathcal{Z}_\ell(f)$). When $\mathcal{R}$ is a square, one should specify in which direction the crossing occurs.
\end{definition}
It is known that given a rectangle $\mathcal{R}$, the probability of the event $\text{Cross}_{\mathcal{E}_\ell(f)}(\mathcal{R})$ depends on $\ell$ and undergoes some sharp transition as $\ell$ crosses the so called critical parameter $\ell_c=0$.
\begin{theorem}[\cite{BG16}, \cite{BM18},\cite{RV19} for $\ell=0$, \cite{RV20},\cite{MV20} for $\ell>0$]
\label{thm:rsw}
Assume that $q$ satisfies Assumptions \ref{a:a1}, \ref{a:a2} for some $\alpha \geq 3$, \ref{a:a3} (weak positivity) and \ref{a:a4} for some $\beta>2$. Let $\mathcal{R}\subset \R^2$ be any non degenerated rectangle.
\begin{enumerate}
    \item If $\ell=0$ and $q$ satisfies Assumption \ref{a:a3} (weak positivity), then there exists a constant $c \in ]0,1[$ depending on $q$ and $\mathcal{R}$ such that
    \begin{equation}
    \label{eq:7}
    \forall \lambda>1,\ \Proba{\text{Cross}_{\mathcal{E}_0(f)}(\lambda \mathcal{R})}\in [c,1-c].
\end{equation}
\begin{equation}
\label{eq:8}
    \forall \lambda>1,\ \Proba{\text{Cross}_{\mathcal{Z}_0(f)}(\lambda \mathcal{R})}\in [c,1-c].
\end{equation}
\item If $\ell>0$ and $q$ satisfies Assumption \ref{a:a3} (strong positivity), then there exists a constant $c>0$ depending on $q$ and $\mathcal{R}$ such that
\begin{equation}
    \forall \lambda>1,\ \Proba{\text{Cross}_{\mathcal{E}_\ell(f)}(\lambda \mathcal{R})} > 1 -\frac{1}{c}e^{-c \lambda}.
\end{equation}
\begin{equation}
    \forall \lambda>1,\ \Proba{\text{Cross}_{\mathcal{Z}_\ell(f)}(\lambda \mathcal{R})}< \frac{1}{c}e^{-c \lambda}.
\end{equation}
\end{enumerate}
Here, $\lambda \mathcal{R}$ denotes the rectangle $\mathcal{R}$ dilated by $\lambda$.
\end{theorem}
We make a few comments about Theorem \ref{thm:rsw}. First, this theorem can be understood as a quantitative description of the transition phase of our percolation model. When $\ell>0$, big rectangles are crossed with high probability by $\mathcal{E}_\ell(f)$, this high connectivity is a witness of the existence of a unique unbounded connected component in $\mathcal{E}_\ell(f)$ (see \cite{MV20} for more details). When $\ell=0$, we are in the so-called \textit{critical regime} which will be the focus of this article. The first item of Theorem \ref{thm:rsw} is sometimes referred to as a Russo-Seymour-Welsh property (RSW) or also as a box-crossing property. It can be interpreted as saying that the probability of having a positive (or negative) crossing in a rectangle remains bounded away from $0$ and $1$ as the rectangle becomes bigger and bigger. Note that this property together with some result of quasi-independence imply that with probability $1$, there is no unbounded component in $\mathcal{E}_0(f)$ (nor in $\mathcal{Z}_0(f)$) (see \cite{BG16} but also \cite{Ale96} for earlier results on this problem). We also briefly comment that if $\mathcal{R}$ is a square, then equation \eqref{eq:7} of Theorem \ref{thm:rsw} remains true without assuming that $q$ satisfies Assumption \ref{a:a3} (weak positivity). In fact, by duality (using the fact that $f$ and $-f$ have the same law) and by $\pi/2$ rotational invariance, one can readily show that we always have $\mathbb{P}(\text{Cross}_{\mathcal{E}_0(f)}(\mathcal{R}))=\frac{1}{2}.$ Moreover, since a crossing of a square by $\mathcal{Z}_0(f)$ implies a crossing of the same rectangle by $\mathcal{E}_0(f)$ we also have for any square $\mathcal{R}$, $\mathbb{P}(\text{Cross}_{\mathcal{Z}_0(f)}(\mathcal{R}))\leq \frac{1}{2}$ which can be thought of as an upper bound of \eqref{eq:8}.

In this paper we are interested in the geometry of the set $\mathcal{E}_\ell(f)$ and $\mathcal{Z}_\ell(f)$ at \textit{criticality} ($\ell=\ell_c=0$). More precisely, our goal is to give a description of the geometry of the connected components of those sets that are of big diameter. In fact, Theorem \ref{thm:rsw} can be interpreted as follows: in a fixed rectangle of scale $\lambda$ there exists a connected component of $\mathcal{E}_0(f)$ (or $\mathcal{Z}_0(f)$) of diameter of order $\lambda$ with a probability that stays bounded away from $0$ as $\lambda$ varies. We prove that those macroscopic connected components (there exists at least one with positive probability) present a fractal behavior with high probability. To be more precise we first introduce some terminology.
\begin{definition}
\label{def:curve}
A \textit{parametrization} of a curve is a continuous function $\gamma : [0,1]\to \R^2$. Two parametrizations $\gamma_1$ and $\gamma_2$ are said to be equivalent if there there exists an homeomorphism $\varphi : [0,1]\to [0,1]$ such that $\gamma_1 = \gamma_2 \circ \varphi.$ A \textit{curve} is an equivalence class for this relation of equivalence. A curve $\mathcal{C}$ is said to be a \textit{rectifiable curve} if there exists a parametrization $\gamma$ of of the curve that is Lipschitz. In that case the \textit{length} of the curve is defined as
\begin{equation}
    \text{length}(\mathcal{C}):= \int_{0}^1 \norm{\gamma'(t)}{}dt,
\end{equation}
where $\norm{\cdot}{}$ denotes the Euclidean distance of $\R^2$. The quantity $\text{length}(\mathcal{C})$ is sometimes referred to as the Euclidean length of the curve.
Since $\gamma$ is Lipschitz, Rademacher's theorem implies that the above integral is well defined and takes value in $\mathbb{R}_+$. Moreover, it can be checked that this integral does not depend on the choice of the Lipschitz parametrization (see \cite{Fal85} for details).
\end{definition}
\begin{remark}
This notion of curve is fairly standard and general enough to allow self intersecting curves. We note in particular that curves that admit a continuous piecewise $\mathcal{C}^1$ parametrization are rectifiable. In the rest of the paper, all curves considered are rectifiable curves even if not mentioned. Moreover, a curve naturally induces a subset of $\R^2$ by considering the set $\{\gamma(t)\ |\ t\in [0,1]\}$ where $\gamma$ is any parametrization of $\mathcal{C}$ (note however that two different curves may induce the same subset of $
\R^2$). When we refer to certain geometric properties of a curve such as the diameter it should be understood as the diameter of this subset of $\R^2$.
\end{remark}
We now define the set of macroscopic curves in a rectangle.
\begin{definition}
\label{def:curves}
Let $\mathcal{R}\subset \R^2$ a non degenerated rectangle, $E\subset \R^2$ be any subset and $\lambda \geq 1$. We denote by $\mathcal{C}(\mathcal{R},E,\lambda)$ the set of continuous and rectifiable curves $\mathcal{C}$ that are included in $\lambda\mathcal{R}\cap E$ and that are of Euclidean diameter at least $\lambda$.
\end{definition}
We state our first theorem.
\begin{theorem}
\label{thm:principal1}
There exists a constant $\beta_0>4$ such that if $q$ satisfies Assumptions \ref{a:a1}, \ref{a:a2} for some $\alpha \geq 3$ and \ref{a:a4} for $\beta > \beta_0$, the following holds. If $E$ denotes either $\mathcal{E}_0(f)$ or $\mathcal{Z}_0(f)$, there exists $s_1>1$ such that for any fixed non degenerated rectangle $\mathcal{R}\subset \R^2$ we have:
\begin{equation}
    \label{eq:thm_principal_1}
    \forall \delta>0,\ \exists C_1>0,\ \forall \lambda>1,\  \Proba{\forall \mathcal{C}\in \mathcal{C}(\mathcal{R},E,\lambda),\ \emph{length}(\mathcal{C})> C_1\lambda^{s_1}}\geq 1-\delta.
\end{equation}
\end{theorem}
Before stating corollaries of Theorem \ref{thm:principal1}, we make a few comments on the hypotheses and the statement of this theorem. First, this theorem should be interpreted by saying that with high probability the macroscopic connected component of $\mathcal{E}_0(f)$ (or $\mathcal{Z}_0(f)$) are very tortuous in the sense that any macroscopic curve drawn on it must travel a distance of order $\lambda^{s_1}\gg \lambda$. This behavior is very different to the behavior in the supercritical phase ($\ell>0$) where previous work (see \cite{Ver23}) shows that with high probability whenever two points are connected in $\mathcal{E}_\ell(f)$ the chemical distance (that is the length of the shortest path connecting them in $\mathcal{E}_\ell(f)$) is at most quasilinear in the Euclidean distance between those two points. In the critical phase however, Theorem \ref{thm:principal1} shows that this cannot be the case. We also make a few technicals remarks on the statement of Theorem \ref{thm:principal1}. First of all, note that the event written in the probability in \eqref{eq:thm_principal_1} may not be measurable with respect to the usual $\sigma$-field generated by the finite collections of values of the field $f$. However \eqref{eq:thm_principal_1} should be interpreted by saying that there exists a measurable event of probability at least $1-\delta$ under which all curves in $\mathcal{C}(\mathcal{R},E,\lambda)$ have length at least $C_1\lambda^{s_1}$. Secondly, we mention that the value of $\beta_0$ we obtain is difficult to track and in general we do not expect our hypotheses on $q$ to be optimal. Nevertheless, we bring the attention of the reader to the fact that Theorem \ref{thm:principal1} itself does not requires that $q$ satisfies Assumption \ref{a:a3}. However, in order to show that $\mathcal{C}(\mathcal{R},E,\lambda)$ is not empty with positive probability one may need to add this hypothesis depending on the shape of the rectangle $\mathcal{R}$ as discussed after Theorem \ref{thm:rsw}. Theorem \ref{thm:principal1} is the adaptation of the result of Aizenman and Burchard in \cite{AB98} that holds for very general systems of random curves. More precisely, the arguments developed in \cite{AB98} are general enough to adapt to the case of continuous rectifiable curves. However, one of the step in \cite{AB98} requires the random field $f$ to behave close to independently in disjoint boxes which is a property that is not easy to verify for our continuous field $f$. In order to prove Theorem \ref{thm:principal1}, we follow the strategy in \cite{AB98}, our main contribution is to deal with lack of independence by using a quasi-independence result for nodal lines of smooth Gaussian fields developed in \cite{RV19}, \cite{BMR20}, it turns out that the result we obtain with this quasi independence result is not strong enough to be directly applied in the proof of Aizenman and Burchard and we do a careful analysis of their proof to solve this problem.

We now collect some corollaries of Theorem \ref{thm:principal1}. One curve of interest in a rectangle $\mathcal{R}$ is a shortest curve in $\mathcal{E}_0(f)$ (or in $\mathcal{Z}_0(f)$) that crosses $\mathcal{R}$ if such a curve exists. We thus introduce the following definition.
\begin{definition}
\label{def:S_E}
Let $\mathcal{R}\subset \R^2$ be a a non degenerated rectangle and $E\subset \R^2$ be a subset. On the event $\text{Cross}_E(\mathcal{R})$ we denote by $S_E(\mathcal{R})$ the infimum of the Euclidean length of all continuous and rectifiable curves that cross the rectangle $\mathcal{R}$ in $E$.
\end{definition}
\begin{remark}
Note that when $E$ denotes one of $\mathcal{E}_0(f)$ or $\mathcal{Z}_0(f)$ then classical regularity arguments (see for instance \cite[Lemma A.9]{RV19}) show that with probability $1$, on the event $\text{Cross}_E(\mathcal{R})$, one can indeed find a continuous and rectifiable curve that realizes the crossing.
\end{remark}
We have the following consequences of Theorem \ref{thm:principal1}.
\begin{corollary}
\label{cor:1_1}
Under the same assumptions than Theorem \ref{thm:principal1}, if $q$ additionally satisfies Assumption \ref{a:a3} (weak positivity) then the following is satisfied. There exists a constant $s_1>1$ such that for any non degenerated rectangle $\mathcal{R}\subset \R^2$,
\begin{align}
    & \forall \delta>0,\ \exists C_1>0,\ \forall \lambda>1,\ \mathbb{P}(S_{\mathcal{E}_0(f)}(\lambda \mathcal{R})>C_1\lambda^{s_1}\ |\ \emph{Cross}_{\mathcal{E}_0(f)}(\lambda \mathcal{R}))\geq 1-\delta.
\\
    & \forall \delta>0,\ \exists C_1>0,\ \forall \lambda>1,\ \mathbb{P}(S_{\mathcal{Z}_0(f)}(\lambda \mathcal{R})>C_1\lambda^{s_1}\ |\ \emph{Cross}_{\mathcal{Z}_0(f)}(\lambda \mathcal{R}))\geq 1-\delta.
\end{align}
\end{corollary}
Moreover, one may avoid assuming Assumption \ref{a:a3} by instead considering rectangles of some fixed shapes.
\begin{corollary}
\label{cor:1_2}
Under the same assumptions than Theorem \ref{thm:principal1}, let $\mathcal{R}=\left[-\frac{1}{2},\frac{1}{2}\right]^2,$ we have a constant $s_1>1$ such that
\begin{equation}
    \forall \delta>0,\ \exists C_1>0,\ \forall \lambda>1,\ \mathbb{P}(S_{\mathcal{E}_0(f)}(\lambda \mathcal{R})>C_1\lambda^{s_1}\ |\ \emph{Cross}_{\mathcal{E}_0(f)}(\lambda \mathcal{R}))\geq 1-\delta.
\end{equation}
Here the crossing of $\lambda \mathcal{R}$ refers to the crossing in the horizontal direction.
\end{corollary}
\begin{remark}
We could also state a similar corollary concerning $\mathcal{Z}_0(f)$ without Assumption \ref{a:a3}. In fact, if one take $\mathcal{R}=[0,1]\times [0,3]$, it can be shown using duality argument together with a quasi independence result, that the probability that $\lambda \mathcal{R}$ is crossed horizontally by $\mathcal{Z}_0(f)$ is bounded away from $0$ as $\lambda$ varies. Note however that this notion of crossing is weaker than the notion defined in Definition \ref{def:cross} since here we consider crossing of the rectangle by joining the big sides of the rectangle instead of the small ones.
\end{remark}
For the convenience of the reader, we briefly mention an example of a random field that does not satisfy Assumption \ref{a:a3} but for which Corollary \ref{cor:1_2} applies. We present the example that was introduced in \cite{BG17} and we refer the reader to this paper for the details on the construction. Consider a compact smooth Riemannian manifold $(M,g)$ of dimension $2$. Consider $\Delta$ its Laplacian and $(\varphi_i)_{i\geq 0}$ an Hilbert orthogonal basis composed of the eigenfunctions of $\Delta$ ($\varphi_i$ is of eigenvalue $\lambda_i$). Consider $\chi : \R_+^* \to \R_+$ a smooth function with compact support containing $1$. Let $(a_i)_{i\geq 0}$ a collection of independent standard Gaussian random variables. For $L>0$, one can define the random field
$$f_{\chi,L} := \sum_{i\geq 0}a_i \chi\left(\frac{\lambda_i}{L}\right)\varphi_i.$$
This field can be interpreted as a smooth cutoff of the infinite sum $\sum_{i\geq 0}a_i\varphi,$ where we mainly keep the eigenfunctions with eigenvalues close to $L$. It can be checked that in normal coordinates near a point $x_0=0$ we have
$$\forall (x,y)\in (\R^2)^2,\ \mathbb{E}\left[f_{\chi,L}\left(\frac{x}{L}\right)f_{\chi,L}\left(\frac{y}{L}\right)\right]\xrightarrow[L\to \infty]{}K_\chi(x,y),$$
where $K_\chi : \R^2 \times \R^2 \to \R$ is a positive definite kernel. Moreover we have
\begin{equation}
    \label{eq:kernel_without_fkg}
    \forall (x,y)\in (\R^2)^2,\ K_\chi(x,y) = \int_{\xi\in \R^2}\chi(\norm{\xi}{}^2)e^{i\scal{x-y}{\xi}}d\xi.
\end{equation}
Note that in \eqref{eq:kernel_without_fkg}, the norm and scalar product denote the usual Euclidean norm and scalar product. The smoothness of $\chi$ and the fact that $\chi$ has compact support ensure that $K_\chi$ has at least super-polynomial decay when $\norm{x-y}{}\xrightarrow[]{}\infty$ and that $K_\chi$ is smooth. However, when $\kappa$ approximates the Dirac distribution $\delta_1$ we see that the kernel $K_\chi$ converges on all compacts to the kernel $K$ associated to the random wave model. That is
\begin{equation}
    \label{eq:random_wave_model}
    \forall (x,y)\in (\R^2)^2,\ K(x,y) = \int_{\xi\in \R^2, \norm{\xi}{}=1}e^{i\scal{x-y}{\xi}}d\xi = J_0(\norm{x-y}{}),
\end{equation}
where $J_0$ denotes the Bessel function. Since $K$ takes negatives values, if one takes $\chi$ a good enough approximation of $\delta_1$ we obtain $K_\chi$ a kernel that decays fast but that does not satisfy Assumption \ref{a:a3}.

While Theorem \ref{thm:principal1} together with Corollaries \ref{cor:1_1} and \ref{cor:1_2} are evidence of the fractal behavior of the macroscopic components of $\mathcal{E}_0(f)$ and $\mathcal{Z}_0(f)$, we now state our second theorem that shows that this behavior cannot be too degenerated.
\begin{theorem}
\label{thm:principal2}
There exists a constant $\alpha_0\geq 3$ such that if $q$ satisfies Assumptions \ref{a:a1}, \ref{a:a2} for some $\alpha \geq \alpha_0$, Assumption \ref{a:a3} (weak positivity) and \ref{a:a4} for $\beta > 2$ the following holds. If $E$ denotes either $\mathcal{E}_0(f)$ or $\mathcal{Z}_0(f)$, there exists $s_2<2$ such that for any non degenerated rectangle $\mathcal{R}\subset \R^2$,
\begin{equation}
    \forall \delta>0,\ \exists C_2>0,\ \forall \lambda >1,\  \Proba{S_E(\lambda \mathcal{R}) < C_2\lambda^{s_2}\ |\ \emph{Cross}_E(\lambda \mathcal{R})}\geq 1-\delta.
\end{equation}
\end{theorem}
Before presenting the strategy of the proof of both Theorem \ref{thm:principal1} and \ref{thm:principal2} we make a few comments.
The fact that one can choose $s_2<2$ in Theorem \ref{thm:principal2} shows that the shortest crossing is not too degenerated when it exists. In fact, a classical application of the Kac-Rice formula shows that the expected length of all the nodal lines contained in $\lambda \mathcal{R}$ is of order $\lambda^2$. Thus, Theorem \ref{thm:principal2} implies that the shortest crossing is far from enough to collect all this length of nodal lines. Indeed, there are also many connected components of $\mathcal{Z}_0(f) \cap \lambda \mathcal{R}$ that are not macroscopic but count towards the total length of order $\lambda^2$. Our next remark is that it is conjectured for Bernoulli percolation that there should by an exponent $1<s<2$ such that the length of the shortest crossing in a box of size $\lambda$ (conditioned on the existence of the crossing) is of order $\lambda^s$. To our knowledge, the determination of the value of $s$ is still an open problem even for Bernoulli percolation (see \cite{Sch11} Problem 3.3). Although recent progress in \cite{DHS17} show that this exponent should be strictly lower than the exponent of the lowest crossing of the box $\lambda \mathcal{R}$ in the case of critical Bernoulli bond percolation on $\Z^2$. To conclude, we comment that the value of $\alpha_0$ obtained is certainly not optimal and in general we do not believe our assumptions on the field to be optimal.

We now describe the strategy of the proof of both Theorems \ref{thm:principal1} and \ref{thm:principal2}.
Concerning Theorem \ref{thm:principal1}, as mentioned earlier the proof is essentially based on the argument developed in \cite{AB98}. Aizenman and Burchard present in \cite{AB98} a very general argument to prove that random curve satisfying some hypothesis present a fractal behavior. The hypothesis they make on the random curve is that the probability that the curve crosses $n$ rectangles that are \textit{well-separated} (see Definition \ref{def:well_separated}) decays geometrically in $n$. Although this is easy to check in the context of Bernoulli percolation (since we have independence), this is not the case in the context of continuous Gaussian fields. We address this problem by using a quasi-independence result for nodal lines developed in \cite{RV19} and \cite{BMR20}. However, we do not manage to obtain the exact hypothesis made by Aizenman and Burchard in \cite{AB98} and we do a careful analysis of their argument in \cite{AB98} to conclude the proof of Theorem \ref{thm:principal1}. For the proof of Theorem \ref{thm:principal2}, we argue that by a result of \cite{BG16}, the number of boxes of size $1$ visited by the curve realising the crossing of $\lambda \mathcal{R}$ is of order at most $\lambda^{2-\eta}$ where $\eta$ is the one-arm exponent (see \eqref{eq:onearmexponent}). We then use a result in previous work \cite{Ver23} allowing us to control the chemical distance in each of these boxes of size $1$.

Since the general strategy of the proof of Theorem \ref{thm:principal1} comes from \cite{AB98},  Section \ref{sec:2} is dedicated to the introduction of some terminology and some results of \cite{AB98} that are adapted to our framework with Gaussian fields. In Section \ref{sec:3}, we use the quasi-independence result of \cite{RV19} and \cite{BMR20} to conclude the proof of  of Theorem \ref{thm:principal1}, we also give the proof of Theorem \ref{thm:principal2}.

\textbf{Acknowledgements : }I would like to thank my PhD advisor Damien Gayet for introducing me to this problem as well as for his remarks on a preliminary version of this paper. I also would like to thank Vincent Beffara for mentioning \cite{AB98} to us.

\section{Deterministic control}
\label{sec:2}
In this section we adapt and restate results of \cite{AB98} to adapt to our framework. In this whole section $\lambda>1$ denotes real parameter, $\mathcal{R}\subset \R^2$ is any non degenerated rectangle, $E\subset \R^2$ can be any subset (typically $E$ can thought of as $\mathcal{E}_0(f)$ or $\mathcal{Z}_0(f)$), $\mathcal{C}$ is a continuous and rectifiable curve included in $\lambda \mathcal{R}$ of diameter at least $\lambda$ (that is an element of $\mathcal{C}(\mathcal{R},E,\lambda)$ with Definition \ref{def:curves}). We begin with a definition.
\begin{definition}
\label{not:1}
A triplet $(m,\gamma,s)\in \mathbb{N}\times \R\times \R$ is said to be a \textit{renormalization triplet} if $s>1$ and if we have $1\leq m< \gamma<\gamma^s < \sqrt{m(m+1)}.$ Given, a renormalization triplet $(m,\gamma,s)$, for any integer $k$ we defined the $k$-th scale as $L_k = \frac{\lambda}{\gamma^k}$. We define $k_\text{max}$ to be the biggest integer $k$ such that $\varepsilon L_k \geq 1$ where $\varepsilon:=\frac{\gamma}{m}-1>0.$
\end{definition}
\begin{remark}
The definition of a renormalization triplet may seem strange at first glance. The reason such a choice was made will be made clear in Proposition \ref{prop:low_energy} and \ref{prop:sparse}. We note however that for any $m\geq 1$ a renormalization triplet $(m,\gamma,s)$ can be built. In fact we can choose some $s>1$ close enough to $1$ so that $m^s<\sqrt{m(m+1)}$, and then choose any $\gamma>m$ close enough to $m$ such that $\gamma^s<\sqrt{m(m+1)}.$ We also remark that most of the results in this section are empty whenever $\lambda <\gamma$, it may therefore be convenient to have in mind big values of $\lambda$ (so that $k_{\text{max}}>1$).
\end{remark}
One should keep in mind that we will apply the result of this section to the random set $E=\mathcal{E}_0(f)$ or $E=\mathcal{Z}_0(f)$. However since the results of this section are completely deterministic we choose to state them with a fixed set $E$ and a fixed curve $\mathcal{C}$.

We introduce a few definitions that come from \cite{AB98}.
\begin{definition}
\label{def:sparse}
In the following, $(m,\gamma,s)$ denotes a renormalization triplet,
\begin{itemize}
    \item A \textit{straight run} of $E$ at some scale $L\in \mathbb{R}_+$ is the data of a rectangle $\mathcal{R}_L$ of length $L$ and of height $\frac{9L}{\sqrt{\gamma}}$ that is crossed by a connected component of $E\cap \mathcal{R}_L$ in the length direction by joining the two centers of the sides of the rectangle (note that the rectangle does not need to be aligned with some specific axes and can be freely rotated).
    \item Two straight runs of $E$ are \textit{nested} if one of the rectangles can be included in the other rectangle.
    \item For $k_0\in \mathbb{N}$, straight runs of $E$ are are said to be \textit{$(\gamma, k_0)$-sparse} in $\lambda \mathcal{R}$ if there does not exist any $n\geq 0$ together with a sequence $1\leq k_1<\dots < k_n\leq k_{\max}$ with $n\geq \frac{1}{2}\max(k_n, k_0)$ such that $\lambda \mathcal{R}$ contains a sequence of nested straight runs of $E$ at scales $L_{k_1}>L_{k_2}>\dots >L_{k_n}.$
\end{itemize}
\end{definition}
We comment on the above definition and we try to motivate it. The idea of Aizenman and Burchard is that a curve that does not present many straight runs is bound to take many detours at many scales hence generating a fractal behavior. If one think of the famous Koch snowflake, we start from a straight unit length segment (that obviously presents a straight run at scale $1$). We cut this segment into three parts and replace the middle one with two bends of length $1/3$. Because of this, the first iteration no longer presents a straight run at scale $1$, however each of the four segments of length $1/3$ presents a straight runs at scale $1/3$. By iterating this procedure we see that ultimately our Koch snowflake will contort a lot and will avoid doing straight runs. Of course when dealing with a general curve (that typically comes from a random set) it is hopeless to hope that straight runs will be avoided at each scale. Thus, the notion of \textit{sparse straight runs} introduced by Aizenman and Burchard is a good way of handling this difficulty. Their strategy can be divided into two parts, the first part is deterministic and consist in saying that if straight runs of $E$ are \textit{sparse} then a curve drawn in $E$ must contort a lot, the second part is probabilistic and aims to control the probability of having sparse straight runs of $E$. In this section we mostly focus on the first part and we begin with the following lemma that comes from \cite{AB98}.
\begin{lemma}[Lemma 5.2 in \cite{AB98}]
\label{lemma:1}
Let $(m,\gamma, s)$ be a renormalization triplet, there exists a sequence $(\Gamma_k)_{0\leq k \leq k_{\text{max}}}$ of collections of segments of the curve $\mathcal{C}$ such that:
\begin{itemize}
    \item for all $k$, each segment $\eta \in \Gamma_k$ is of diameter at least $L_k$,
    \item for all $k$, any two segments $\eta_1,\eta_2\in \Gamma_k$ are at distance at least $\varepsilon L_k$ where $\varepsilon=\frac{\gamma}{m}-1$,
    \item for all $k\geq 1$, any segment $\eta$ of $\Gamma_k$ is included in a segment $\eta'\in \Gamma_{k-1}$ (we say that $\eta$ is a child of $\eta'$),
    \item for all $k\leq k_{\text{max}}-1$, any segment $\eta\in \Gamma_k$ has at least $m$ children. Moreover if there is no straight run of $\eta$ at scale $L_k$ then $\eta$ has at least $(m+1)$ children,
    \item we can assume that $\Gamma_0$ contains only one initial segment.
\end{itemize}
\end{lemma}
We refer the reader to \cite{AB98} for the proof of the above lemma as it is a purely algorithmic construction. We briefly comment on the interpretation of Lemma \ref{lemma:1} for our purpose. At each step any segment of the curve of diameter $L_k$ is split via an algorithmic procedure into at least $m$ segments of diameter $L_{k+1}=\frac{L_k}{\gamma}$. Moreover, when the segment contorts and avoid doing a straight runs, then this segment will be divided into at least $m+1$ smaller segments. If straight runs of $E$ are sparse in $\mathcal{R}$ then this last case will happened at least half of the times when considering a trajectory from an ancestor to its descendants, hence the average number of children of a segment should be at least $\sqrt{m(m+1)}$. This will be the key to prove that the curve $\mathcal{C}$ presents a fractal behavior since this average number of children is bigger than the scale $\gamma$ of renormalization.
We now present the tool Aizenman and Burchard introduced to measure the tortuosity of a curve. It is the notion of the \textit{energy} of the curve.
\begin{definition}
\label{def:emu}
Let $(m,\gamma,s)$ be a renormalization triplet. Let $\mu$ be a probability measure supported on $\mathcal{C}$. The \textit{energy} $E_s(\mu)$ is defined as
\begin{equation}
\label{eq:defemu}
    E_s(\mu):= \lambda^s\iint_{(x,y)\in \mathcal{C}\times \mathcal{C}}\frac{\mu(dx)\mu(dy)}{\max(|x-y|^s,1)}.
\end{equation}
\end{definition}
The relation between the tortuosity of the curve and the energy $E_s(\mu)$ can be seen in the following lemma.
\begin{lemma}[Lemma 5.3 in \cite{AB98}]
\label{lemma:h}
Let $(m,\gamma,s)$ be a renormalization triplet. Let $(C_i)_{1\leq i \leq N}$ be a finite collection of subsets of $\R^2$ such that
\begin{itemize}
    \item The $C_i$ are disjoint.
    \item For all $i$, $\text{diam}(C_i)\geq 1.$
    \item $\mathcal{C}=\bigsqcup_{i=1}^N C_i.$
\end{itemize}
Let $\mu$ be a probability measure supported on $\mathcal{C}.$
We have
\begin{equation}
    E_s(\mu)\sum_{i=1}^N \frac{\emph{diam}(C_i)^s}{\lambda^s} \geq 1,
\end{equation}
\end{lemma}
As we slightly modified the definition of the energy from \cite{AB98} we provide the adapted proof of this lemma here for the convenience of the reader.
\begin{proof}[Proof of Lemma \ref{lemma:h}]
Since $\mathcal{C}$ is partitioned by the $C_i$ we write
\begin{align*}
    E_s(\mu)& \geq \sum_{i=1}^N \lambda^s\iint_{(x,y)\in C_i\times C_i}\frac{\mu(dx)\mu(dy)}{\max(|x-y|^s,1)} \\
    &\geq \lambda^s\sum_{i=1}^N \iint_{(x,y)\in C_i\times C_i} \frac{\mu(dx)\mu(dy)}{\text{diam}(C_i)^s} \\
    &= \lambda^s\sum_{i=1}^N \frac{\mu(C_i)^2}{\text{diam}(C_i)^s}.
\end{align*}
Since $\mu$ is a probability measure supported on $\mathcal{C}$ (partitioned by the $C_i$) we have
$$1=\sum_{i=1}^N \mu(C_i).$$
Thus applying the Cauchy-Schwarz inequality we get
\begin{align*}
    1 = \left(\sum_{i=1}^N \mu(C_i)\right)^2 \leq \left(\sum_{i=1}^N \frac{\lambda^s\mu(C_i)^2}{\text{diam}(C_i)^s}\right)\left(\sum_{i=1}^N \frac{\text{diam}(C_i)^s}{\lambda^s}\right) \leq E_s(\mu) \sum_{i=1}^N \frac{\text{diam}(C_i)^s}{\lambda^s}.
\end{align*}
This is precisely the conclusion.
\end{proof}
Since we are interested in the length of the curve $\mathcal{C}$ we now relate the energy $E_s(\mu)$ to this length.
\begin{corollary}
\label{cor:1}
Let $(m,\gamma,s)$ be a renormalization triplet, for any probability measure $\mu$ supported on $\mathcal{C}$ we have
\begin{equation}
    \emph{\text{length}}(\mathcal{C})\geq \frac{\lambda^s}{E_s(\mu)}-2^s
\end{equation}
\end{corollary}
\begin{proof}
Without loss of generality, considering a sub-curve of $\mathcal{C}$ of same diameter but that does not self intersect, we may assume that we have $\gamma : [0,1]\to \R^2$ a continuous and injective parametrization of $\mathcal{C}$. We recursively define, $t_0=0, x_0 =\gamma(t_0)$, and for $n\geq 0$, if $t_n$ is defined we define $t_{n+1}$ as follows. Either $x_n = \gamma(t_n)$ is at distance less than $1$ from $\gamma(1)$, we then set $t_{n+1}=1$, $x_{n+1}=\gamma(1)$ and we end the procedure. Otherwise, we define $t_{n+1} := \text{min}(t>t_n\ |\ d(\gamma(t),\gamma(t_n))=1\}$.
Applying this procedure yields a finite sequence $(t_i)_{0\leq i \leq N+1}$ with $0=t_0<t_1<\dots<t_{N}<t_{N+1}=1.$
We define for $1\leq i \leq N-1$, $C_i = \gamma([t_{i-1},t_i[).$
We also set $C_N := \gamma([t_{N-1},1]).$
By construction we see that the $C_i$ are disjoint, cover $\mathcal{C}$, moreover for all $1\leq i\leq N-1$ we have $\text{diam}(C_i)=1$ and $\text{diam}(C_N) \in [1,2[.$ Hence we can apply Lemma \ref{lemma:h} to see that for any probability measure $\mu$ supported on $\mathcal{C}$ we have
$$\sum_{i=1}^N \text{diam}(C_i)^s \geq \frac{\lambda^s}{E_s(\mu)}.$$
Now note that for all $1\leq i \leq N-1$ we have $\text{diam}(C_i)^s = \text{diam}(C_i)=1$, and we always have $\text{diam}(C_i)\leq \text{length}(\gamma([t_{i-1},t_i[))$. Moreover using the fact that $\text{diam}(C_N)\leq 2$ we get
\begin{align*}
    \text{length}(\mathcal{C}) = \sum_{i=1}^N \text{length}(C_i) \geq \sum_{i=1}^{N}\text{diam}(C_i) \geq -2^s +\sum_{i=1}^N \text{diam}(C_i)^s \geq \frac{\lambda^s}{E_s(\mu)}-2^s.
\end{align*}
\end{proof}
By the previous corollary, in order to show that $\mathcal{C}$ has a length of order at least $\lambda^s$ one need to find $\mu$ a probability measure on $\mathcal{C}$ such that $E_s(\mu)$ is upper bounded. This is what is done in the following proposition.
\begin{proposition}[see Lemma 5.4 in \cite{AB98}]
\label{prop:low_energy}
Let $(m,\gamma,s)$ be a renormalization triplet. Assume that straight runs of $E$ are $(\gamma,k_0)$-sparse in $\lambda\mathcal{R}$, then there exists a probability measure $\mu$ supported on $\mathcal{C}$ with low energy, that is $\mu$ satisfies
$$E_s(\mu) \leq \frac{1}{\varepsilon^s}\left(\gamma^{s(k_0+1)}+\frac{\beta}{1-\frac{\gamma^s}{\beta}}\right),$$
where $\beta$ is defined as $\beta := \sqrt{m(m+1)}$ (note that by definition of $s$ we have $\gamma^s < \beta$).
\end{proposition}
The proof of Proposition \ref{prop:low_energy} is completely similar to the one in \cite{AB98}. However, since we modified the definition of the energy we provide the adapted proof for the convenience of the reader.
\begin{proof}
Define a measure $\mu$ on $\mathcal{C}$ as follows, the measure gives weight $1$ to the unique segment in $\Gamma_0$, then each segment divides its weight evenly among its children until we reach the segments in $\Gamma_{k_{\text{max}}}.$ More formally, given an integer $0\leq k \leq k_{\text{max}}$ and some $i< k$ one can define $n_i(\eta)$ as the number of children of the unique segment $\eta'\in \Gamma_i$ that contains $\eta$. For any segment $\eta \in \Gamma_k$ we thus define
\begin{equation}
    \mu(\eta) = \prod_{i=0}^{k-1}\frac{1}{n_i(\eta)}.
\end{equation}
For convenience, we arbitrarily choose one point in each $\eta\in \Gamma_{k_{\text{max}}}$ and decide that the measure $\mu$ is supported on the finite collection of those points (for instance we take the leftmost upmost point of each $\eta\in \Gamma_{k_{\text{max}}}$). We denote $\tilde{\mathcal{C}}$ the collection of those points, so that $\mu$ is supported on $\tilde{\mathcal{C}}$.

The fact that straight runs of $E$ are $(\gamma,k_0)$-sparse in $\lambda \mathcal{R}$ implies the following claim.
\begin{claim}
\label{claim:1}
For any $x\in \tilde{\mathcal{C}}$, for any $k$ such that $k_0\leq k \leq k_{\text{max}}$, we have
\begin{equation}
    \prod_{i=0}^{k-1} n_i(x) \geq \beta^k,
\end{equation}
where we recall that $\beta =\sqrt{m(m+1)}$.
\end{claim}
\begin{proof}[Proof of the Claim \ref{claim:1}]
Note that $n_i(\eta)$ is always greater or equal than $m$, and is strictly greater than $m$ if there is no straight run at scale $L_i$ for the segment $\eta'\in \Gamma_i$ containing $\eta$.
By contradiction assume that for some $k_0\leq k \leq k_{\text{max}}$ we have strictly more than half of the factors in the product $\prod_{i=0}^{k-1}n_i(x)$ that are equal to $m$. That is we can find a $n>\frac{k}{2}$ and a sequence $0\leq k_1<\dots <k_n \leq k-1$, such that for all $1\leq i \leq n$ we have $n_{k_i}(x) = m$. This would imply that there exist a nested straight run on the scales $L_{k_1},\dots,L_{k_n}$ this is a contradiction with the fact that straight runs are $(\gamma,k_0)$-sparse (in fact we have $n>\frac{k}{2}$ so $n\geq \frac{1}{2}\max(k_n,k_0)$).
We now see that at least half of the factors in the product are greater or equal than $m+1$ and the other are always greater or equal than $m$. This readily yields the conclusion.
\end{proof}
We go back to the proof of Proposition \ref{prop:low_energy}. Given two points $x,y\in \tilde{\mathcal{C}}$, denote $k(x,y)$ the first generation where $x$ and $y$ were not part of the same segment. By convention if $x=y$ we arbitrarily define $k(x,y)=k_{\text{max}}+1$. By the definition of the $\Gamma_k$ we see that for $x\neq y$ we have $d(x,y) \geq \varepsilon L_{k(x,y)}$.

The energy $E_s(\mu)$ can be estimated ad follows:

\begin{align*}
    \lambda^{-s}E_s(\mu) &= \iint_{(x,y)\in \tilde{\mathcal{C}}}\frac{\mu(dx)\mu(dy)}{\max(|x-y|^s,1)} \\
    &\leq \sum_{k=1}^{k_{\text{max}}+1} \iint_{(x,y) | k(x,y)=k}\frac{\mu(dx)\mu(dy)}{\max(|x-y|^s,1)} \\
    &\leq \sum_{x\in \tilde{\mathcal{C}}}\mu(\{x\})^2+\sum_{k=1}^{k_{\text{max}}}\frac{1}{(\varepsilon L_k)^s}\iint_{(x,y) | k(x,y)=k}\mu(dx)\mu(dy).
\end{align*}
Now note, that if $x,y\in \tilde{\mathcal{C}}$ are such that $k(x,y)=k$ then $x$ and $y$ belong to the one common segment $\eta\in \Gamma_{k-1}$, this yields
$$\iint_{(x,y)|k(x,y)=k}\mu(dx)\mu(dy) \leq \sum_{\eta \in \Gamma_{k-1}}\iint_{(x,y)\in \eta\times \eta}\mu(dx)\mu(dy) \leq \sum_{\eta \in \Gamma_{k-1}}\mu(\eta)^2.$$
Also by definition of $k_{\text{max}}$, we have $L_{k_{\text{max}}+1} \leq \frac{1}{\varepsilon}$ which implies $1\leq \frac{1}{(\varepsilon L_{k_{\text{max}}+1})^s}$. We thus can write
\begin{equation}
    \label{eq:Emu1}
    \lambda^{-s}E_s(\mu)\leq \sum_{k=1}^{k_{\text{max}}+1}\frac{1}{(\varepsilon L_k)^s}\sum_{\eta \in \Gamma_{k-1}}\mu(\eta)^2.
\end{equation}
For $k> k_0$ we observe that
\begin{align*}
    \sum_{\eta\in \Gamma_k}\mu(\eta)^2 & = \sum_{\eta\in \Gamma_k}\mu(\eta) \prod_{i=0}^{k-1}n_i(\eta) \\
    &\leq \sum_{\eta\in \Gamma_k}\mu(\eta) \frac{1}{\beta^k}\\
    &= \frac{1}{\beta^k},
\end{align*}
where the first inequality is a direct application of Claim \ref{claim:1} and the last equality is due to the fact that $\mu$ is a probability measure (hence $\sum_{\eta\in \Gamma_k}\mu(\eta)=1$).
For $k\leq k_0$ we simply have
\begin{align*}
    \sum_{\eta \in \Gamma_k}\mu(\eta)^2 \leq \sum_{\eta\in \Gamma_k}\mu(\eta)\leq 1.
\end{align*}

Now we split the sum in \eqref{eq:Emu1} according to the values of $k$ and we use the fact that $L_k=\frac{\lambda}{\gamma^k}$ and the fact that $1<\gamma<\gamma^s<\beta$:
\begin{align*}
    \lambda^{-s}E_s(\mu) &\leq \sum_{k=1}^{k_0}\frac{1}{(\varepsilon L_k)^s} + \sum_{k=k_0+1}^{k_{\text{max}}}\frac{1}{(\varepsilon L_k)^s}\frac{1}{\beta^{k-1}} \\
    &\leq \frac{1}{\varepsilon^s}\frac{1}{\lambda^s}\sum_{k=1}^{k_0}(\gamma^s)^k + \frac{\beta}{\varepsilon^s}\frac{1}{\lambda^s}\sum_{k=0}^\infty \left(\frac{\gamma^s}{\beta}\right)^k \\
    &\leq \frac{1}{\varepsilon^s}\frac{1}{\lambda^s}\left(\gamma^{s(k_0+1)}+\frac{\beta}{1-\frac{\gamma^s}{\beta}}\right).
\end{align*}
This concludes
\begin{equation}
    E_s(\mu)\leq \frac{1}{\varepsilon^s}\left(\gamma^{s(k_0+1)}+\frac{\beta}{1-\frac{\gamma^s}{\beta}}\right).
\end{equation}
\end{proof}

\section{Proof of the main theorems}
\label{sec:3}
In this section we prove Theorem \ref{thm:principal1} and Theorem \ref{thm:principal2}, we separate the proof of the two theorems and present the proof of the lower bound first.
\subsection{Proof of the lower bound}
In the following we use a quasi-independence result to show that with high probability straight runs of $E$ are $(\gamma,k_0)$-sparse in $\lambda \mathcal{R}$ where $E$ denotes either $\mathcal{E}_0(f)$ or $\mathcal{Z}_0(f)$ (see Definition \ref{def:sparse}).
We begin by introducing the notion of sets that are \textit{well-separated}.
\begin{definition}
\label{def:well_separated}
Let $(A_i)_{1\leq i}$ a collection of subset of $\R^2$ we say that this collection is \textit{well-separated} if for all $i$ we have
$d(A_i, \bigcup_{j\neq i}A_j) \geq \text{diam}(A_i),$
where $d$ denotes the usual Euclidean distance.
\end{definition}
In \cite{AB98}, the following Assumption was made on the random curve $\mathcal{C}$.
\begin{assumption}
\label{a:AB}
There exists $\sigma>0$ and $0<q<1$ such that for any collection $(R_i)_{1\leq i \leq n}$ of well-separated rectangles such that each rectangle is of length as least $1$ and of aspect ratio $\sigma$,
$$\Proba{\bigcap_{i=1}^n\text{Cross}_{\mathcal{C}}(R_i)}\leq q^n,$$
where we recall Definition \ref{def:cross} of the crossing event.
\end{assumption}
It turns out that Assumption \ref{a:AB} is easy to verify for Bernoulli percolation since the configurations on disjoint set are independent. However, in the context of our smooth Gaussian percolation, such an assumption is not easy to verify. We will use a quasi-independence result for nodal lines (see \cite{RV19},\cite{BMR20}) to obtain a weaker version of Assumption \ref{a:AB} that will be enough to conclude. First we make the following definition
\begin{definition}
Let $f$ be a centered stationary Gaussian field on $\R^2$ of covariance function $\kappa : \R^2 \to \R$. We define the function $\tilde{\kappa}$ as 
$$\begin{array}{ccccc}
    \tilde{\kappa} & : & \R_+ & \to & \R  \\
     &  & x & \mapsto & \sup\{\kappa(y)\ |\ y\in \R^2, \norm{y}{}\geq x\}
\end{array}.$$
\end{definition}
\begin{remark}
For instance, for the Bargmann Fock field, since we have $\forall y\in \R^2, \ \kappa(y)=e^{-\frac{1}{2}\norm{y}{}^2}$ we deduce that $\forall x\in \R_+, \ \tilde{\kappa}(x)=e^{-\frac{x^2}{2}}.$
\end{remark}
We now provide an analogue of Assumption \ref{a:AB}.
\begin{lemma}
\label{lemma:quasi_indep}
If $q$ satisfies Assumption \ref{a:a1}, \ref{a:a2} for some $\alpha\geq 3$ and \ref{a:a4} for some $\beta>4$ the following holds. Let $\sigma\geq 1$, $l_0>0$ and $(\mathcal{R}_i)_{i\geq 1}$ be a collection of well-separated rectangles in the plane $\R^2$. We assume that $\mathcal{R}_i$ has length $\sigma l_i$ and height $l_i$. Moreover we assume that the sequence $(l_i)_{i\geq 1}$ is increasing and that $l_1\geq l_0$. Let $E$ denotes either $\mathcal{E}_0(f)$ or $\mathcal{Z}_0(f)$. There exists a constant $C>0$ (depending only on $q,\sigma$ and $l_0$) such that for all $n\geq 0$
\begin{equation}
    \label{eq:indep_el}
    \Proba{\bigcap_{i=1}^n \emph{Cross}_{E}(\mathcal{R}_i)} \leq \frac{1}{2^n}\left(1+C\sum_{k=1}^n l_k^4\tilde{\kappa}(l_k)k2^k\right),
\end{equation}
\end{lemma}
\begin{proof}
We do the proof for $E=\mathcal{E}_0(f)$ as the proof for $E=\mathcal{Z}_0(f)$ is completely similar.
For $n\geq 0$, denote $p_n := \Proba{\bigcup_{i=1}^n \text{Cross}_{\mathcal{E}_0(f)}(\mathcal{R}_i)}.$
Applying Theorem 1.12 of \cite{RV19} or Theorem 2.14 of \cite{BMR20} (see also Corollary 1.1 of \cite{BMR20}).
We see that we have a constant $C>0$ (depending on $q$, $\sigma$ and $l_0$) such that
$$p_n \leq \Proba{\text{Cross}_{\mathcal{E}_0(f)}(\mathcal{R}_n)}p_{n-1}+C\tilde{\kappa}(l_n)(l_n+1)^2\sum_{i=1}^{n-1}(l_i+1)^2.$$
Observe that since $\sigma\geq 1$ we have $\Proba{\text{Cross}_{\mathcal{E}_0(f)}(R_n)}\leq \frac{1}{2}$. This simply comes from the fact that $f$ and $-f$ have the same law since $f$ is a centered Gaussian field and using the fact that the law of $f$ is invariant by rotation of $\pi/2$. (this also allows us to deduce that $\Proba{\text{Cross}_{\mathcal{Z}_0(f)}(R_n)}\leq \frac{1}{2}$).
Thus, we have a constant $C'>0$ (depending on $q$, $\sigma$ and $l_0$) such that,
$$p_n \leq \frac{1}{2}p_{n-1}+C'n\tilde{\kappa}(l_n)l_n^4.$$
By telescopage, with $p_0 = 1$, we have for $n\geq 0$
$$2^np_n \leq C'\sum_{k=1}^n k2^k\tilde{\kappa}(l_k)l_k^4 + 1.$$
This rewrites as
$$p_n \leq \frac{1}{2^n}\left(1+C'\sum_{k=1}^n k2^k\tilde{\kappa}(l_k)l_k^4\right),$$
which concludes the proof.
\end{proof}
Now, we are ready to prove that with high probability the straight runs of $E$ are sparse in $\lambda\mathcal{R}$. The proof of the next proposition is our main contribution towards the proof of Theorem \ref{thm:principal1} as we do a careful analysis of the proof of Aizenman and Burchard to show that we can replace the use of Assumption \ref{a:AB} by Lemma \ref{lemma:quasi_indep}.
\begin{proposition}
\label{prop:sparse}
There exists $\beta_0>4$ and there exists a renormalization triplet $(m,\gamma,s)$ such that the following holds.
Let $E$ denoted either $\mathcal{E}_0(f)$ or $\mathcal{Z}_0(f)$ and let $\mathcal{R}\subset \R^2$ be a non degenerated rectangle. If $q$ satisfies Assumption \ref{a:a1}, \ref{a:a2} for $\alpha\geq 3$ and \ref{a:a4} for some $\beta>\beta_0$, then for all $\delta>0$, there exists $k_0\geq 1$ such that for all $\lambda>1$ we have
\begin{equation}
    \Proba{\text{straight runs of }E \text{ are }(\gamma,k_0)\text{-sparse in }\lambda \mathcal{R} }\geq 1 - \delta.
\end{equation}
\end{proposition}
\begin{proof}
We do the proof for $E=\mathcal{E}_0(f)$ as the proof for $\mathcal{Z}_0(f)$ is the same.
Let $(m,\gamma,s)$ be a renormalization triplet that is free for now and let $\lambda>1$ be a free parameter.
Fix $\mathbf{k}=(k_i)_{1\leq i \leq n}$ a sequence such that $1\leq k_1<\dots<k_n\leq k_{\text{max}}$ with $n\geq \frac{1}{2}\max(k_n,k_0)$. Denote $\mathcal{A}_{\mathbf{k}}$ the event that in $\lambda \mathcal{R}$ there exists a sequence of nested straight runs of $E$ at scales $L_{k_1}>L_{k_2}>\dots >L_{k_n}.$ On the event $\mathcal{A}_{\mathbf{k}}$ one can find nested rectangles $\mathcal{R}_1,\dots, \mathcal{R}_n$ such that $\mathcal{R}_i$ is of length $L_{k_i}$, of height $\frac{9L_{k_i}}{\sqrt{\gamma}}$ and such that some connected component of $E\cap \mathcal{R}_n$ joins the two centers of the sides of $\mathcal{R}_i$.

For $j\geq 1$, denote by $\mathbb{L}_j$ the lattice $\mathbb{L}_j := \frac{L_{k_j}}{\gamma}\mathbb{Z}^2$. Since $E$ crosses the rectangle $\mathcal{R}_1$, it also crosses a rectangle $\mathcal{R}'_1$ of length $\frac{L_{k_1}}{2}$ and height $\frac{10L_{k_1}}{\sqrt{\gamma}}$ centered on a line segment joining two points of $\mathbb{L}_1\cap \lambda\mathcal{R}$, the number of possibilities of such a rectangle $\mathcal{R}'_i$ is at most $C_\mathcal{R}\left(\frac{\lambda}{L_{k_1}/\gamma}\right)^2 = C_\mathcal{R}\gamma^{2(k_1+1)},$ where $C_\mathcal{R}>0$ is a constant depending on $\mathcal{R}.$ Similarly, since $E$ crosses $\mathcal{R}_2\subset \mathcal{R}_1$ in the length direction, one can find a rectangle $\mathcal{R}'_2$ of length $\frac{L_{k_2}}{2}$ and of height $\frac{10L_{k_2}}{\sqrt{\gamma}}$ centered on a segment joining two points in $\mathbb{L}_2 \cap \mathcal{R}_1$. The number of possibilities for such a rectangle is at most $\left(\frac{L_{k_1}}{L_{k_2}/\gamma}\right)^2=\gamma^{2(k_2-k_1)+2}.$ We can repeat this procedure to build rectangles $\mathcal{R}'_i$ for $1\leq i \leq n$. The total number of possibilities for the position of the $(\mathcal{R}'_i)_{1\leq i \leq n}$ is at most
$$C_\mathcal{R}\gamma^{2k_1+2}\gamma^{2(k_2-k_1)+2}\dots \gamma^{2(k_n-k_{n-1})+2} = C_\mathcal{R}\gamma^{2(n+k_n)}.$$
Moreover, the crossing of the rectangle $\mathcal{R}_i'$ in the length direction implie that there are least $\frac{\sqrt{\gamma}}{40}$ disjoint smaller rectangles of height $\frac{10L_{k_i}}{\sqrt{\gamma}}$ and of length $\frac{20L_{k_i}}{\sqrt{\gamma}}$ which are all crossed in the length direction. Among the collection of those smaller rectangles (for all $i$) one can extract a family of well-separated rectangles such that for all $i$ there are at least $M:=\frac{\sqrt{\gamma}}{80}-2$ rectangles of dimensions $\frac{20L_{k_i}}{\sqrt{\gamma}} \times \frac{10L_{k_i}}{\sqrt{\gamma}}$ in this collection. This is done as follows: at each scale $L_{k_i}$ take one of the smaller rectangle over two and remove the two such rectangles closest to the rectangle $\mathcal{R}'_{i+1}$. In the following, denote by $(R_j)_{1\leq j \leq N}$ this collection of well-separated rectangles where $N\geq nM = n\left(\frac{\sqrt{\gamma}}{80}-2\right)$. By the previous construction one may assume that $R_j$ is of length $2l_j$ and of height $l_j$ for some $l_j \geq 0$. Moreover, we can assume that the sequence $l_j$ is increasing and that we have $$l_j \geq \frac{10\gamma^{\frac{j}{M}-1}}{\sqrt{\gamma}},$$
where we recall that $M=\frac{\sqrt{\gamma}}{80}-2.$ We remark that we have chosen to work rectangle of size $2l_j \times l_j$, but the same argument would also work if one were to chose squares of size $l_j \times l_j$, the reason we choose to work with those rectangle $R_i$ is simply for convenience so that we avoid having to specify the direction in which the crossing occurs in the squares.

Note that since $q(x) = O(\norm{x}{}^{-\beta})$ this implies $\kappa(x)=O(\norm{x}{}^{-\beta})$ and then $\tilde{\kappa}(x) = O(x^{-\beta})$ We write $\beta = 4+r$ with some $r>0$.
Doing an union bound on all possible positions for the rectangles $(R_i)_{1\leq i \leq n}$ and applying Lemma \ref{lemma:quasi_indep} and using the fact that $\tilde{\kappa}(x)=O(x^{-4-r})$ we have a constant $C_\gamma>0$ (depending on $q$ and $\gamma$) such that
\begin{align*}
    \Proba{\mathcal{A}_{\mathbf{k}}} &\leq C_{\mathcal{R}}\frac{\gamma^{2(n+k_n)}}{2^N}\left(1+C_\gamma\sum_{k=1}^N \frac{k2^k}{l_k^r}\right)\\
    &\leq C_\mathcal{R}\frac{\gamma^{2(n+k_n)}}{2^N}\left(1+C_{r,\gamma}\sum_{k=1}^N k\left(\frac{2}{\gamma^{r/M}}\right)^k\right) \\
    &\leq \frac{C_{\mathcal{R},r,\gamma}\gamma^{2(n+k_n)}}{2^N}\left(1+N\left(\frac{2}{\gamma^{r/M}}\right)^N\right),
\end{align*}
Here $C_{\mathcal{R},r,\gamma}$ is a positive constant that depends only on $\mathcal{R}$ and the values of $\gamma$ and $r$. This constant may change from line to line. However note that this constant does not depend on $\lambda$. In the following, depending on the value of $\gamma$ (to be determined) later, we choose $r=r(\gamma)$ big enough so that
\begin{equation}
    \label{eq:rgamma}
    \frac{2}{\gamma^{r/M}}\leq \frac{1}{2}.
\end{equation}
Note that this is possible as soon as $r\geq \frac{\ln(4)M}{\ln(\gamma)}$. Since we have $M=\frac{\sqrt{\gamma}}{80}-2$ one may note that $r(\gamma)$ goes to infinity as $\gamma$ goes to infinity.
Under our condition \eqref{eq:rgamma}, we have
\begin{equation}
    \label{eq:Ak1}
    \Proba{\mathcal{A}_{\mathbf{k}}} \leq \frac{C_{\mathcal{R},r,\gamma}\gamma^{2(n+k_n)}}{2^N}\left(1+\frac{N}{2^N}\right).
\end{equation}
We can slightly modify the constant $C_{\mathcal{R},r,\gamma}$ in \eqref{eq:Ak1} and recall that we have $N\geq n\left(\frac{\sqrt{\gamma}}{80}-2\right)$ to get
\begin{equation}
    \label{eq:Ak2}
    \Proba{\mathcal{A}_{\mathbf{k}}} \leq \frac{C_{r,\gamma,\mathcal{R}}\gamma^{2(n+k_n)}}{2^{n\left(\frac{\sqrt{\gamma}}{80}-2\right)}}.
\end{equation}

Now for $k\geq 1$, consider the event $\mathcal{B}_k$ that there exists a nested straight run in $\lambda \mathcal{R}$ for some sequence $\mathbf{k}=(k_i)_{1\leq i \leq n}$ with $1\leq k_1<k_2<\dots <k_n <k_{\text{max}}$ and with the additional condition that $k=k_n\geq n\geq \frac{k}{2}.$  By an union bound we have
\begin{align*}
    \Proba{\mathcal{B}_k}\leq \sum_{n=\frac{k}{2}}^k \sum_{\mathbf{k} | k_n = k}\Proba{\mathcal{A}_{\mathbf{k}}}.
\end{align*}
where the second sum is over the sequences $\mathbf{k}=(k_i)_{1\leq i \leq n}$ with $1\leq k_1<k_2<\dots <k_n$ with $k_n=k\geq n \geq \frac{k}{2}.$ Note that for fixed $k$ there are at most $\binom{k}{n}$ such sequences of length $n$. Using our estimate \eqref{eq:Ak2} on the probability of $\mathcal{A}_k$ one get that under condition \eqref{eq:rgamma} we have
\begin{equation}
    \label{eq:Bk1}
    \Proba{\mathcal{B}_k} \leq \sum_{n=\frac{k}{2}}^k \binom{k}{n} \frac{C_{\mathcal{R},r,\gamma}\gamma^{2(n+k)}}{2^{n\left(\frac{\sqrt{\gamma}}{80}-2\right)}}.
\end{equation}
We deduce from \eqref{eq:Bk1} that we have
\begin{equation}
    \label{eq:Bk2}
    \Proba{\mathcal{B}_k} \leq C_{\mathcal{R},r,\gamma}2^k\gamma^{4k}2^{-k\left(\frac{\sqrt{\gamma}}{160}-1\right)}=C_{\mathcal{R},r,\gamma}e^{k\left(\ln(2)+4\ln(\gamma)+1-\frac{\sqrt{\gamma}}{160}\right)}.
\end{equation}
Now we choose a renormalization triplet $(m,\gamma,s)$ with $m\in \mathbb{N}$ big enough so that the value of $\gamma>m$ satisfies
\begin{equation}
    \label{eq:cond_gamma}
    \ln(2)+4\ln(\gamma)+1-\frac{\sqrt{\gamma}}{160}<0.
\end{equation}
For this specific renormalization triplet, we choose $r$ big enough so that \eqref{eq:rgamma} is satisfied (this defines the value of $\beta_0>4$). Under these conditions, we denote
\begin{equation}
    \rho:= e^{\ln(2)+4\ln(\gamma)+1-\frac{\sqrt{\gamma}}{160}}.
\end{equation}
We have $0<\rho<1$ and we have a constant $C_{\mathcal{R},r,\gamma}$ such that:
\begin{equation}
    \Proba{\mathcal{B}_k}\leq C_{\mathcal{R},r,\gamma}\rho^k.
\end{equation}
Finally, observe that if straight runs of $E$ are not $(\gamma,k_0)$-sparse in $\lambda \mathcal{R}$ then one of the event $\mathcal{B}_k$ must occur for some $k\geq k_0$. Hence by an union bound we have
\begin{equation}
    \Proba{\{\text{straight runs of }E\text{ are not }(\gamma,k_0)\text{-sparse in }\lambda \mathcal{R}\}}\leq C_{\mathcal{R},r,\gamma}\frac{\rho^{k_0}}{1-\rho}.
\end{equation}
This probability can be made smaller than $\delta>0$ as soon as $k_0$ is big enough (note that there are not dependence on $\lambda$ in neither $\rho$ nor $C_{\mathcal{R},r,\gamma}$). This concludes the proof.
\end{proof}
We now state the proof of Theorem \ref{thm:principal1}.
\begin{proof}[Proof of Theorem \ref{thm:principal1}]
Let $\beta_0>4$ and $(m,\gamma,s)$ the renormalization triplet given by Proposition \ref{prop:sparse} and be fixed for the rest of the proof. In the following, we assume that $q$ satisfies Assumption \ref{a:a4} for $\beta>\beta_0$. Let $E$ denotes either $\mathcal{E}_0(f)$ or $\mathcal{Z}_0(f)$ and $\mathcal{R}\subset \R^2$ be a fixed rectangle. Let $\delta>0$, by Proposition \ref{prop:sparse}, we can find $k_0\geq 1$ (depending only on $\delta$, $q$ and $\mathcal{R}$) such that for all $\lambda>1$ we have
\begin{equation}
    \Proba{\text{straight runs of }E\text{ are }(\gamma,k_0)\text{-sparse in }\lambda \mathcal{R}} \geq 1-\delta.
\end{equation}
On the event that the straight runs of $E$ are $(\gamma, k_0)$-sparse in $\lambda \mathcal{R}$, given any curve $\mathcal{C}\in \mathcal{C}(\mathcal{R},E,\lambda)$ (that is a continuous rectifiable curve in $\lambda \mathcal{R}\cap E$ of diameter at least $\lambda$, see Definition \ref{def:curves}) we can apply Proposition \ref{prop:low_energy} to find a probability measure $\mu$ supported on $\mathcal{C}$ such that its energy $E_s(\mu)$ satisfies
\begin{equation}
    E_s(\mu)\leq \frac{1}{\varepsilon^s}\left(\gamma^{s(k_0+1)}+\frac{\beta}{1-\frac{\gamma^s}{\beta}}\right).
\end{equation}
Where we recall that $\varepsilon=\frac{\gamma}{m}-1$ and $\beta=\sqrt{m(m+1)}.$ Thus, applying Corollary \ref{cor:1}, we see that we have a constant $c>0$ depending on $k_0$ as well than on the renormalization triplet $(m,\gamma,s)$ such that
\begin{equation}
    \text{length}(\mathcal{C}) \geq c\lambda^s.
\end{equation}
This can be reformulated by saying that on the event that the straight runs of $E$ are $(\gamma, k_0)$-sparse in $\lambda \mathcal{R}$, we have
$$\forall \mathcal{C}\in \mathcal{C}(\mathcal{R},E,\lambda),\ \text{length}(\mathcal{C})\geq c\lambda^s.$$
This precisely yields the conclusion of Theorem \ref{thm:principal1}.
\end{proof}
We conclude this section with the proofs of Corollaries \ref{cor:1_1} and \ref{cor:1_2}.
\begin{proof}[Proof of Corollary \ref{cor:1_1}]
Without loss of generality we assume that the rectangle $\mathcal{R}$ is of unit length. We let $E$ denote either $\mathcal{E}_0(f)$ or $\mathcal{Z}_0(f)$. Let $\delta>0$ and take $C_1$ associated by Theorem \ref{thm:principal1}. Let $\lambda>1$, recall that $\mathcal{C}(\mathcal{R},E,\lambda)$ denotes the set of continuous rectifiable curves in $E\cap \lambda \mathcal{R}$ of diameter at least $\lambda$. In particular, any continuous and rectifiable curve that realizes the crossing of $\lambda\mathcal{R}$ in the length direction belongs to $\mathcal{C}(\mathcal{R},E,\lambda).$ If $\mathcal{G}$ denotes the event
$$\mathcal{G}:=\left\{\forall \mathcal{C}\in \mathcal{C}(\mathcal{R},E,\lambda),\ \text{length}(\mathcal{C})\geq C_1\lambda^{s_1}\right\},$$
then we have
$$\mathcal{G}\cap \text{Cross}_E(\lambda \mathcal{R})\subset \{S_E(\lambda\mathcal{R})\geq C_1\lambda^{s_1}\}\cap \text{Cross}_E(\lambda \mathcal{R}).$$
Thus,
\begin{equation}
    \label{eq:proof_cor1}
    \Proba{\{S_E(\lambda\mathcal{R})\geq C_1\lambda^{s_1}\}\cap \text{Cross}_E(\lambda \mathcal{R})}\geq \Proba{\text{Cross}_E(\lambda \mathcal{R})}-\delta.
\end{equation}
It remains to divide both sides of \eqref{eq:proof_cor1} by $\Proba{\text{Cross}_E(\lambda \mathcal{R})}$ that stays bounded away from $0$ as $\lambda$ varies by the first item of Theorem \ref{thm:rsw}.
\end{proof}
\begin{proof}[Proof of Corollary \ref{cor:1_2}]
The proof is exactly the same as the proof of Corollary \ref{cor:1_1} until we reach \eqref{eq:proof_cor1}. Then instead of applying Theorem \ref{thm:rsw} to show that $\Proba{\text{Cross}_{\mathcal{E}_0(f)}(\lambda \mathcal{R})}$ stays bounded away from $0$, we simply use the symmetry of the field ($f$ and $-f$ have the same law) together with the invariance of the law of $f$ by rotation of $\pi/2$ to get $\Proba{\text{Cross}_{\mathcal{E}_0(f)}(\lambda \mathcal{R})}=\frac{1}{2}.$
\end{proof}

\subsection{Proof of the upper bound}
We now turn to the proof of Theorem \ref{thm:principal2}. As mentioned in the introduction, the idea is to control the number of boxes of size $1$ visited by the curve realizing the shortest crossing. In order to do so we recall a result of \cite{BG16} about the \textit{one-arm event}.

\begin{definition}
Let $1\leq s<t$ be two real parameters and $E$ denotes either $\mathcal{E}_0(f)$ or $\mathcal{Z}_0(f)$. The \textit{one-arm event} denoted by $\text{Arm}_E^1(s,t)$ is the event that there exists a connected component of $E$ that intersects both $\left[-\frac{s}{2},\frac{s}{2}\right]^2$ and the boundary of $\left[-\frac{t}{2},\frac{t}{2}\right]^2$.
The probability of $\text{Arm}_E^1(s,t)$ is denoted by $\pi_E^1(s,t).$
\end{definition}
We have the following crucial result.
\begin{theorem}[Theorem 1.4 of \cite{BG16}, see also \cite{MV20} for $\beta>2$]
\label{thm:one_arm}
There exists a constant $\eta>0$, such that, if $q$ satisfies Assumptions \ref{a:a1}, \ref{a:a2} for some $\alpha\geq 3$, \ref{a:a3} (weak positivity) and \ref{a:a4} for some $\beta>2$ then we have a constant $C>0$ such that:
\begin{equation}
    \label{eq:onearmexponent}
    \forall 1\leq s\leq t,\  \pi_E^1(s,t) \leq C\left(\frac{s}{t}\right)^\eta,
\end{equation}
where $E$ denotes either $\mathcal{E}_0(f)$ or $\mathcal{Z}_0(f)$.
\end{theorem}
We also recall definitions and a result of a previous work in \cite{Ver23}.
\begin{definition}
\label{def:SAB}
Let $B\subset \R^2$ be a square box of with $1$. Let $A\subset \R^2$ be a subset. We assume that $B\cap A$ has finitely many connected components : $B\cap A = \bigsqcup_{r=1}^q C_r$ where the $C_r$ are connected. We make the following definitions.
\begin{enumerate}
    \item For $x,y\in C_r$, the \textit{chemical distance} between $x$ and $y$ is denoted by $d_\chem(x,y)$ and is defined as the minimal length of a continuous rectifiable path $\gamma$ joining $x$ and $y$ within $C_r$.
    \item The \textit{chemical diameter} of $C_r$ is denoted by $\text{diam}_\chem(C_r)$ and is defined as
    $$\text{diam}_\chem(C_r):= \sup\{d_\chem(x,y)\ |\ x,y\in C_r\}.$$
    \item We define the quantity $S(A,B)$ as
    $$S(A,B) := \sum_{r=1}^q \text{diam}_\chem(C_r).$$
\end{enumerate}
\end{definition}
\begin{remark}
Typically, one should have in mind that the set $A$ is either $\mathcal{E}_\ell(f)$ or $\mathcal{Z}_\ell(f)$, in that case the quantity $S(A,B)$ as well as the notion of chemical distance and chemical diameter are random variables.
\end{remark}
We now state the result that allows us to control the chemical distance within a box.
\begin{proposition}[Proposition 3.7 in \cite{Ver23}]
\label{prop:moment}
Assume that $q$ satisfies Assumptions \ref{a:a1}, \ref{a:a2} for some $\alpha\geq 3$, \ref{a:a4} for some $\beta> 2$. Let $B\subset \R^2$ be a square box of side-length $1$. Then for any $k\leq \alpha-1$ there exists a constant $M_k\in \mathbb{R}_+$ such that
\begin{align}
    \mathbb{E}[S(\mathcal{E}_0(f),B)^k]\leq M_k, \\
    \mathbb{E}[S(\mathcal{Z}_0(f),B)^k]\leq M_k.
\end{align}
\end{proposition}
We make a few comments about Proposition \ref{prop:moment}. The proof of this proposition relies on a deterministic argument allowing us to compare the chemical distance with the length of the nodal lines in a box together with a result of \cite{GS23}, \cite{AL23} concerning the finiteness of moments for the length of nodal lines. Although in \cite{Ver23}, Proposition \ref{prop:moment} was stated only for $\mathcal{E}_0(f)$ it is straightforward (and much easier) to adapt the proof for $\mathcal{Z}_0(f)$.

We now do the proof of the upper bound of Theorem \ref{thm:principal2}.

\begin{proof}[Proof of Theorem \ref{thm:principal2}]
Let $E$ denotes either $\mathcal{E}_0(f)$ or $\mathcal{Z}_0(f)$. We consider $\mathcal{R}\subset \R^2$ a fixed rectangle and $\lambda>1$ a real parameter. 
We prove that there exists $s_2<2$ such that with high probability on the event $\text{Cross}_E(\lambda \mathcal{R})$ we have $S_E(\lambda \mathcal{R}) \leq C\lambda ^{s_2}$ (recall Definition \ref{def:S_E} of $S_E$). Let $(B_j)_{1\leq j\leq N}$ with $N\leq C_1\lambda^2$ a covering of $\lambda \mathcal{R}$ by $N$ square boxes of size $1$ ($C_1>0$ is a constant depending only on $\mathcal{R}$). We write $B_j = b_j+\left[-\frac{1}{2},\frac{1}{2}\right]^2$ where $b_j \in \R^2$ is the center of $B_j$.
For $1\leq j \leq N$ denote $\mathcal{A}_j\subset \text{Cross}_E(\lambda \mathcal{R})$ the event that $B_j$ intersects some continuous rectifiable curve $\mathcal{C}$ realizing the crossing of $\lambda\mathcal{R}$ in $E$ (not necessarily the shortest one). We then define
\begin{equation}
    \mathcal{N} := \sum_{j=1}^N \mathds{1}(\mathcal{A}_j),
\end{equation}
the number of boxes $B_j$ that are visited by such curves on the event $\text{Cross}_E(\lambda \mathcal{R})$. Note that since a curve $\mathcal{C}$ that realizes the crossing of $\lambda\mathcal{R}$ is a curve of diameter at least $\lambda$, if a box $B_j$ is visited by such a curve then there is a path in $E$ joining $B_j$ to the boundary of $b_j+\left[-\frac{\lambda}{2},\frac{\lambda}{2}\right]$. Applying Theorem \ref{thm:one_arm} and using stationarity we see that we have a constant $C_2>0$ such that
\begin{equation}
    \forall \lambda>1, \Proba{\mathcal{A}_j} \leq \frac{C_2}{\lambda^\eta}.
\end{equation}
where $\eta>0$ is the fixed constant introduced in Theorem \ref{thm:one_arm}. In particular by linearity of expectation we have
\begin{equation}
    \label{eq:N}
    \mathbb{E}[\mathcal{N}]\leq N\frac{C_2}{\lambda^\eta}\leq C_1C_2\lambda^{2-\eta}=C_3\lambda^{2-\eta},
\end{equation}
where $C_3>0$ is a constant depending on $\mathcal{R}.$
Now we want to argue that in each box $B_j$ visited some curve realizing the crossing, one can control the time that an optimal curve should spend in this box. This is done using Proposition \ref{prop:moment}. Let $\nu := \frac{\eta}{2}$. With the notations of Definition \ref{def:SAB}, for $1\leq j \leq N$ we introduce the event $\mathcal{B}_j := \{S(E,B_j)\leq C_4\lambda^{\nu}\}$ where we recall that $E$ denotes either $\mathcal{E}_0(f)$ or $\mathcal{Z}_0(f)$ and where $C_4>0$ denotes a constant to be fixed later. Note that for all $k\leq m-1$ by Proposition \ref{prop:moment} together with the Markov inequality we have a constant $C_5^{(k)}>0$ (depending on $k$) such that
\begin{equation}
    \forall \lambda>1,\ \Proba{\mathcal{B}_j}\geq 1-\frac{C_5^{(k)}}{(C_4\lambda)^{\nu k}}.
\end{equation}
Now we choose $\alpha_0$ big enough so that $\nu(\alpha_0-1)\geq 2$, and we assume that $q$ verifies Assumption \ref{a:a2} for some $\alpha \geq \alpha_0$. Doing an union bound on all $1\leq j\leq N$ we see that
\begin{equation}
    \forall \lambda>1,\ \Proba{\bigcap_{i=1}^N \mathcal{B}_j} \geq 1-\frac{C_1C_5^{(\alpha_0-1)}}{C_4^{\nu(\alpha_0-1)}}.
\end{equation}
Let $\delta>0$, we now fix $C_4>0$ (depending on $\nu, \delta, C_5^{(\alpha_0-1)}$) big enough so that we have
\begin{equation}
    \forall \lambda>1,\ \Proba{\bigcap_{i=1}^N \mathcal{B}_j} > 1-\frac{\delta}{2}.
\end{equation}
Also using \eqref{eq:N} together with a Markov inequality we can find a constant $C_6>0$ (depending on $\delta$ and $\mathcal{R}$) such that
\begin{equation}
    \forall \lambda>1,\ \Proba{\mathcal{N}\leq C_6\lambda^{2-\eta}}> 1-\frac{\delta}{2}.
\end{equation}
Now on the event $\text{Cross}_E(\lambda \mathcal{R}) \cap \{\mathcal{N}\leq C_6\lambda^{2-\eta}\} \cap \bigcap_{j=1}^N B_j$, one can find a continuous rectifiable curve $\mathcal{C}\subset E$ crossing the rectangle $\lambda \mathcal{R}$ in the length direction and that is of Euclidean length less than $C_6\lambda^{2-\eta}C_4\lambda^\nu=C_7\lambda^{2-\nu}$ where $C_7>0$ is a constant (that depends on $\delta, \mathcal{R}$ and on $q$ but that does not depend on $\lambda$) and where we recall that $\nu=\frac{\eta}{2}>0$.
Hence we have proved the following
\begin{equation}
    \forall \lambda>1,\ \Proba{\text{Cross}_E(\lambda \mathcal{R}) \cap \{S_E(\lambda \mathcal{R})\leq C_7\lambda^{2-\nu}\}}\geq \Proba{\text{Cross}_E(\lambda \mathcal{R})}-\delta. 
\end{equation}
Now we can divide by $\Proba{\text{Cross}_E(\lambda \mathcal{R})}$ that is bounded away from $0$ as $\lambda$ varies by Theorem \ref{thm:rsw}. Adjusting constants we precisely get the conclusion of Theorem \ref{thm:principal2}.
\end{proof}
\bibliography{biblio.bib}
\end{document}